\documentclass[11pt]{article}

\usepackage[margin=1.05in]{geometry}
\usepackage{amsmath,amssymb,amsthm,mathtools}
\usepackage{microtype}
\usepackage[hidelinks]{hyperref}
\hypersetup{
  pdftitle={Discrete Concavity of Token-Graph Spectral Radii via Lorentzian Semigroups},
  pdfauthor={Weiqi Jiang},
  pdfsubject={Spectral concavity and Lorentzian heat flow for weighted token graphs},
  pdfkeywords={token graph, spectral radius, discrete concavity, Lorentzian polynomial, ultra log-concavity, Lie--Trotter product formula}
}

\pdftrailerid{%
  \pdfmdfivesum file {main.tex}%
  \pdfmdfivesum file {references.bib}%
}

\newtheorem{theorem}{Theorem}[section]
\newtheorem{lemma}[theorem]{Lemma}
\newtheorem{corollary}[theorem]{Corollary}
\newtheorem{remark}[theorem]{Remark}

\newcommand{\R}{\mathbb{R}}
\newcommand{\one}{\mathbf{1}}
\newcommand{\lmax}{\lambda_{\max}}
\newcommand{\Sym}{\operatorname{Sym}}

\title{Discrete Concavity of Token-Graph Spectral Radii\\
via Lorentzian Semigroups}
\author{Weiqi Jiang\thanks{Institute of Theoretical Physics, Chinese Academy of Sciences;
\href{mailto:jiangweiqi@itp.ac.cn}{\texttt{jiangweiqi@itp.ac.cn}}}}
\date{August 23, 2026}

\begin{document}
\maketitle

\begin{abstract}
Let $F_k(G)$ be the $k$-token graph of a finite graph with nonnegative edge
weights, and let $A_k$ and $D_k$ be its weighted adjacency and degree
matrices.  For every $-1\leq\vartheta\leq1$, we prove that
$k\mapsto\lambda_{\max}(A_k+\vartheta D_k)$ is discretely concave.
Complement symmetry then makes this sequence nondecreasing up to the middle
level.  At $\vartheta=1$ and $\vartheta=0$, this gives the signless-Laplacian
and adjacency spectral-radius monotonicity conjectures of Apte, Parekh, and
Sud.  The spectral result follows from a finite-time theorem: for every
$t\geq0$, the heat contents
$\binom nk^{-1}\one^{\!*}e^{t(A_k+\vartheta D_k)}\one$ are log-concave in
$k$.  We encode all token levels in one Lorentzian polynomial.  A
four-variable operator symbol proves preservation by each edge heat gate,
and the Lie--Trotter formula passes this preservation to the full semigroup.
Large-time growth rates recover the top eigenvalues.  The same construction
also yields Lorentzian polynomials from top spectral projections.  Finally,
we show that the local symbol certifies exactly the parameter range $[-1,1]$.
\end{abstract}



\section{Introduction}

Let $G$ be a finite graph.  A configuration of $k$ indistinguishable tokens
with hard-core exclusion is a $k$-subset of $V(G)$.  One legal move replaces
an occupied vertex by an unoccupied neighbor; the graph of these
configurations and moves is the $k$-token graph $F_k(G)$.  Token graphs were
introduced systematically by Fabila-Monroy et
al.~\cite{fabila-monroy-et-al-2012}.  They interpolate between the original
graph, since $F_1(G)\cong G$, and its higher-particle exclusion dynamics.
They also carry an involution $S\mapsto V(G)\setminus S$, which identifies
levels $k$ and $n-k$.  These elementary features make token graphs a natural
setting for comparing spectra across particle number.

The comparison studied here arose explicitly in work of Apte, Parekh, and
Sud (APS) on block decompositions of $2$-local Hamiltonians
\cite{apte-parekh-sud-2026}.  Their Conjecture~5 asks whether the largest
signless-Laplacian eigenvalue is nondecreasing from $F_k(G)$ to
$F_{k+1}(G)$ before the middle level; Conjecture~6 asks the same for the
adjacency spectral radius.  The conjectures allow nonnegative edge weights.
APS also formulated four fixed-level eigenvalue bounds, their Conjectures
1--4.  Bakshi, Basu, Kothari, and Li (BBKL) proved those four bounds using
Kikuchi-graph methods \cite{bakshi-basu-kothari-li-2026}.  The BBKL result
does not state or prove the successive-level comparisons in APS
Conjectures~5 and~6.  Thus the latter problems call for a mechanism that
couples all particle-number blocks, rather than an estimate within one
block.

We obtain that mechanism for an entire interval of matrices.  Let
$G=(V,E,w)$ have $n$ vertices and nonnegative edge weights $w_e$.  Write
$A_k$ and $D_k$ for the weighted adjacency and degree matrices of
$F_k(G)$, and put
\begin{equation}\label{eq:H-definition}
 H_k^{(\vartheta)}=A_k+\vartheta D_k,
 \qquad
 h_k^{(\vartheta)}=\lmax\bigl(H_k^{(\vartheta)}\bigr).
\end{equation}
We include the one-vertex levels $k=0,n$, for which
$H_0^{(\vartheta)}=H_n^{(\vartheta)}=[0]$.

\begin{theorem}[Spectral concavity]\label{thm:spectral}
For every nonnegatively weighted finite graph and every
$-1\leq\vartheta\leq1$,
\begin{equation}\label{eq:spectral-concavity}
 2h_k^{(\vartheta)}
 \geq h_{k-1}^{(\vartheta)}+h_{k+1}^{(\vartheta)}
 \qquad (1\leq k\leq n-1).
\end{equation}
Furthermore,
$h_k^{(\vartheta)}=h_{n-k}^{(\vartheta)}$, and hence
\begin{equation}\label{eq:monotonicity}
 h_k^{(\vartheta)}\leq h_{k+1}^{(\vartheta)}
 \qquad (0\leq k<n/2).
\end{equation}
\end{theorem}

Taking $\vartheta=1$ in \eqref{eq:monotonicity} proves Conjecture 5 of
\cite{apte-parekh-sud-2026}, because $H_k^{(1)}=D_k+A_k$ is the signless
Laplacian.  Taking $\vartheta=0$ proves their Conjecture 6 for adjacency
spectral radii.

Theorem~\ref{thm:spectral} follows from the following finite-time ultra
log-concavity statement.

\begin{theorem}[Finite-time ultra log-concavity]\label{thm:heat}
For $t\geq0$ and $-1\leq\vartheta\leq1$, set
\begin{equation}\label{eq:heat-content}
 a_k^{(\vartheta)}(t)
 =\one_k^{\!*}e^{tH_k^{(\vartheta)}}\one_k.
\end{equation}
Then, for $1\leq k\leq n-1$,
\begin{equation}\label{eq:ulc}
 \left(\frac{a_k^{(\vartheta)}(t)}{\binom nk}\right)^2
 \geq
 \frac{a_{k-1}^{(\vartheta)}(t)}{\binom n{k-1}}
 \frac{a_{k+1}^{(\vartheta)}(t)}{\binom n{k+1}}.
\end{equation}
\end{theorem}

At the other endpoint,
\begin{equation}\label{eq:minus-one}
 H_k^{(-1)}=A_k-D_k=-L(F_k(G)).
\end{equation}
Here $L(F_k(G)):=D_k-A_k$ is the weighted token-graph Laplacian.
The positive semidefiniteness of the weighted graph Laplacian gives
$h_k^{(-1)}=0$.  Moreover, $L(F_k(G))\one_k=0$ whether or not the token graph
is connected, so
\begin{equation}\label{eq:minus-one-heat}
 e^{-tL(F_k(G))}\one_k=\one_k,
 \qquad a_k^{(-1)}(t)=\binom nk.
\end{equation}
Thus equality holds in the finite-time inequalities, as well as in the
spectral inequalities, at $\vartheta=-1$.

Theorem~\ref{thm:heat} is stronger than the spectral conclusion: it is a
coefficient inequality at every finite time, before a limiting eigenvalue is
selected.  It also holds without connectivity or strict positivity of the
edge weights.  Together with Theorem~\ref{thm:spectral}, our other main
consequence is that a canonical polynomial built from the top spectral
projection at every fixed token level is Lorentzian; see
Corollary~\ref{cor:perron-polynomial}.

Here is the proof architecture.  We represent a vector indexed by subsets
as a multiaffine polynomial.  The product $\prod_i(y+z_i)$ simultaneously
initializes the all-ones vector on every token level, with $y$ homogenizing
the different subset degrees.  The full generator is a sum of edge
generators.  On the two states in which exactly one endpoint of an edge is
occupied, a local generator is $\vartheta I+P$, where $P$ exchanges the
occupations.  The Lorentzian operator symbol of the associated heat gate is
a nonnegative quadratic.  Its Hessian has eigenvalues
\begin{equation}\label{eq:intro-spectrum}
 1+e^{(\vartheta+1)s},\quad
 1-e^{(\vartheta+1)s},\quad
 -1-e^{(\vartheta-1)s},\quad
 e^{(\vartheta-1)s}-1.
\end{equation}
For $s\geq0$ there is at most one positive eigenvalue throughout
$-1\leq\vartheta\leq1$; for $s>0$, that parameter range is exact.
Br{\"a}nd{\'e}n and Huh's operator-symbol theorem therefore makes each edge
heat gate a Lorentzian preserver \cite{branden-huh-2020}.  The Lie--Trotter
product formula assembles the gates, cone closedness passes to the limit,
and a two-variable specialization gives Theorem~\ref{thm:heat}.  Finally, a
shifted Perron argument identifies each large-time exponential rate with
$h_k^{(\vartheta)}$.

We place this result relative to earlier token-graph spectral work.  Dalf{\'o}
et al. proved, with multiplicities, containment of the Laplacian spectrum of
$F_h(G)$ in that of $F_k(G)$ when $1\leq h\leq k\leq n/2$
\cite{dalfo-et-al-2021}.  Lew bounded the Laplacian eigenvalues newly
appearing between consecutive levels, where ``new'' includes an increase in
multiplicity \cite{lew-2024}.  Dalf{\'o}, Fiol, and Messegu{\'e} obtained
conditional adjacent-level bounds for algebraic connectivity and bounds on
Laplacian eigenvalues absent from the base graph
\cite{dalfo-fiol-messegue-2025}.  These are Laplacian results; they do not
give the signless-Laplacian or adjacency-radius comparisons considered here.
On the adjacency side, Barik and Verma established fixed-level spectral
radius bounds \cite{barik-verma-2024}, while Reyes, Dalf{\'o}, and Fiol gave
fixed-level estimates involving vertex-deleted subgraphs
\cite{reyes-dalfo-fiol-2024}.  Neither is a general comparison from level
$k$ to level $k+1$.

Conjecture~3.3 of Reyes--Dalf{\'o}--Fiol concerns newly appearing Laplacian
eigenvalues and is distinct from APS Conjecture~6.  More precisely, their
$\rho_k(G)$ is the largest eigenvalue newly appearing in the Laplacian
spectrum at level $k$, not the adjacency spectral radius of $F_k(G)$.  We are
not aware of an earlier proof of APS
Conjecture~5 or~6.  Theorem~\ref{thm:spectral} proves the weighted inequalities
formulated as APS Conjectures~5 and~6 and establishes discrete concavity
throughout $-1\leq\vartheta\leq1$;
Theorem~\ref{thm:heat} and Corollary~\ref{cor:perron-polynomial} give the
finite-time and top-projection strengthenings.

\section{Weighted token matrices and Lorentzian preliminaries}

Throughout, $G=(V,E,w)$ is a finite undirected loopless graph with
$V=[n]$ and edge weights $w_e\geq0$.  An edge of weight zero can equivalently
be omitted.  For $0\leq k\leq n$, let
\[
 \binom Vk=\{S\subseteq V:|S|=k\}.
\]
The weighted token graph $F_k(G)$ has this set as its vertex set.  If
$S,T\in\binom Vk$, then the $S,T$ entry of its adjacency matrix is
\begin{equation}\label{eq:weighted-adjacency}
 (A_k)_{S,T}=
 \begin{cases}
  w_{\{i,j\}},&S\mathbin\triangle T=\{i,j\}\in E,\\
  0,&\text{otherwise}.
 \end{cases}
\end{equation}
Thus one endpoint of $\{i,j\}$ lies in $S$ and the other does not.  Its
weighted degree is
\begin{equation}\label{eq:weighted-degree}
 d_k(S)=\sum_{\substack{e\in E\\|S\cap e|=1}}w_e,
 \qquad D_k=\operatorname{diag}\bigl(d_k(S):S\in\tbinom Vk\bigr).
\end{equation}
All matrices are real symmetric.  We write $\one_k$ for the all-ones vector
in $\R^{\binom nk}$; at $k=0,n$ this vector has one entry.

Complementation is already visible at the matrix level.  Let $C_k$ be the
permutation matrix taking the basis vector indexed by $S$ to the one indexed
by $V\setminus S$.  Since
$|S\cap e|=1$ if and only if $|(V\setminus S)\cap e|=1$, one has
\begin{equation}\label{eq:complement-matrices}
 C_kA_kC_k^{\!*}=A_{n-k},\qquad
 C_kD_kC_k^{\!*}=D_{n-k},\qquad
 C_k\one_k=\one_{n-k}.
\end{equation}
These identities hold even when $G$ or a token graph is disconnected.

For $S\subseteq V$, write $z^S=\prod_{i\in S}z_i$.  The linear
identification
\begin{equation}\label{eq:subset-polynomial-identification}
 (c_S)_{S\subseteq V}\longleftrightarrow
 \sum_{S\subseteq V}c_Sz^S
\end{equation}
turns vectors indexed by subsets into multiaffine polynomials.  For an edge
$e=\{i,j\}$, define the degree-preserving local operator
\begin{equation}\label{eq:local-generator}
 L_e^{(\vartheta)}z^S=
 \begin{cases}
  \vartheta z^S+z^{S\mathbin\triangle e},& |S\cap e|=1,\\
  0,& |S\cap e|\in\{0,2\}.
 \end{cases}
\end{equation}
The symmetric difference $S\mathbin\triangle e$ exchanges the occupied and
unoccupied endpoints.  Consequently an edge contributes a token swap and
one crossing-degree term precisely when its endpoints have different
occupations.  On homogeneous multiaffine degree $k$, equations
\eqref{eq:weighted-adjacency}--\eqref{eq:weighted-degree} therefore give
\begin{equation}\label{eq:slice-identification}
 \left.\sum_{e\in E}w_eL_e^{(\vartheta)}\right|_{\deg k}
 =H_k^{(\vartheta)}.
\end{equation}

We next record precisely the Lorentzian tools used in the proof.  We use the
terminology and normalization of Br{\"a}nd{\'e}n and Huh
\cite{branden-huh-2020}.  Their Proposition~2.2 says that a homogeneous
stable polynomial with nonnegative coefficients is Lorentzian; in
particular, a product of linear forms with nonnegative coefficients is a
basic source of Lorentzian polynomials.  Theorem~2.10 gives preservation
under nonnegative linear substitutions.  Corollary~2.11 gives preservation
under nonnegative directional derivatives.  Theorem~2.25 identifies the
Lorentzian cone as the closure of the strictly Lorentzian cone, so the cone
is closed in every fixed homogeneous coefficient space.  Theorem~2.30 and
Corollary~2.32, together, give closure under products.  We will use all of
these statements only for homogeneous polynomials with nonnegative
coefficients.

For a multi-index $\kappa\in\mathbb N^N$, let $\R_\kappa[z]$ denote the
polynomials whose degree in $z_i$ is at most $\kappa_i$.  If a linear
operator $T:\R_\kappa[z]\to\R_\gamma[z]$ is homogeneous, its bounded-degree
symbol is
\begin{equation}\label{eq:operator-symbol}
 \Sym_T(z,u)=
 \sum_{0\leq\alpha\leq\kappa}
 \binom{\kappa}{\alpha}T(z^\alpha)u^{\kappa-\alpha}.
\end{equation}
Here $u^{\kappa-\alpha}=\prod_i u_i^{\kappa_i-\alpha_i}$ and
$\binom\kappa\alpha=\prod_i\binom{\kappa_i}{\alpha_i}$.  The operator-symbol
theorem \cite[Theorem~3.2]{branden-huh-2020} says, in the form needed here,
that a homogeneous operator whose symbol is Lorentzian maps Lorentzian
polynomials in its bounded domain to Lorentzian polynomials.  The degree
bounds matter: below, each occupation variable has bound one, while the
single homogenizing variable has bound $n$.

For a homogeneous quadratic with nonnegative coefficients, the degree-two
characterization following \cite[Theorem~2.25]{branden-huh-2020} reduces
Lorentzianity to its Hessian having at most one positive eigenvalue.  Finally,
Example~2.26 of the same paper says that a bivariate homogeneous polynomial
\[
 \sum_{k=0}^n a_kx^ky^{n-k}
\]
with positive coefficients is Lorentzian precisely when the normalized
sequence $(a_k/\binom nk)_{k=0}^n$ is log-concave.  More generally the
coefficients must have no internal zeros; strict positivity below makes this
condition automatic.

The extra variable $y$ used below is not cosmetic.  A monomial $z^S$ has
degree $|S|$, so a sum over all subsets is not homogeneous.  Replacing it by
$y^{n-|S|}z^S$ puts every token level in total degree $n$.  Thus we work in
the coordinatewise bounded space for variables $(y,z_1,\ldots,z_n)$ with
\begin{equation}\label{eq:kappa-prelim}
 \kappa=(n,1,\ldots,1).
\end{equation}
The local heat gates preserve total degree and respect these bounds, which
is exactly the setting in which Theorem~3.2 applies.

\section{The local heat gate}

Fix an edge $e=\{i,j\}$ and $s\geq0$.  On the two-site multiaffine basis
$(1,z_i,z_j,z_iz_j)$, equation~\eqref{eq:local-generator} reads
\begin{equation}\label{eq:local-basis-generator}
 L_e^{(\vartheta)}(1)=0,\qquad
 L_e^{(\vartheta)}(z_iz_j)=0,\qquad
 \begin{pmatrix}L_e^{(\vartheta)}z_i\\
                 L_e^{(\vartheta)}z_j\end{pmatrix}
 =\begin{pmatrix}\vartheta&1\\1&\vartheta\end{pmatrix}
  \begin{pmatrix}z_i\\z_j\end{pmatrix}.
\end{equation}
Thus the generator vanishes on the zero- and two-occupation states.  On the
ordered one-occupation basis $(z_i,z_j)$ it is $\vartheta I+P$, where
$P^2=I$.  Consequently
\begin{equation}\label{eq:local-exponential}
 e^{s(\vartheta I+P)}
 =e^{\vartheta s}\bigl(\cosh(s)I+\sinh(s)P\bigr)
 =\begin{pmatrix}\alpha&\beta\\\beta&\alpha\end{pmatrix},
\end{equation}
where
\begin{equation}\label{eq:alpha-beta}
 \alpha=e^{\vartheta s}\cosh s,
 \qquad \beta=e^{\vartheta s}\sinh s.
\end{equation}
Both coefficients are nonnegative for $s\geq0$, regardless of the sign of
$\vartheta$.  Exponentiating the other two basis vectors gives the complete
action of the gate $T_e^{(\vartheta)}(s)=e^{sL_e^{(\vartheta)}}$:
\begin{equation}\label{eq:local-basis-gate}
 \begin{aligned}
 T_e^{(\vartheta)}(s)(1)&=1,&
 T_e^{(\vartheta)}(s)(z_iz_j)&=z_iz_j,\\
 T_e^{(\vartheta)}(s)(z_i)&=\alpha z_i+\beta z_j,&
 T_e^{(\vartheta)}(s)(z_j)&=\beta z_i+\alpha z_j.
 \end{aligned}
\end{equation}

\begin{lemma}[Local preserver]\label{lem:local}
For every $s\geq0$ and $-1\leq\vartheta\leq1$, the heat gate
$T_e^{(\vartheta)}(s)$, extended to act as the identity on $y$ and on all
occupation variables outside $e$, preserves Lorentzian polynomials in the
bounded homogeneous space \eqref{eq:kappa-prelim}.
\end{lemma}

\begin{proof}
With $u_i,u_j$ paired to $z_i,z_j$, insert the four formulas in
\eqref{eq:local-basis-gate} into \eqref{eq:operator-symbol}.  The input
monomials $1,z_i,z_j,z_iz_j$ contribute, in that order,
$u_iu_j$, $T(z_i)u_j$, $T(z_j)u_i$, and $T(z_iz_j)$.  Hence the exact
two-site symbol is
\begin{equation}\label{eq:local-symbol}
 \Psi_{\vartheta,s}
 =u_i u_j+z_i z_j
 +\alpha(z_i u_j+z_j u_i)
 +\beta(z_i u_i+z_j u_j).
\end{equation}
In the variable order $(z_i,z_j,u_i,u_j)$, its Hessian is
\begin{equation}\label{eq:hessian}
 \mathcal H_{\vartheta,s}=
 \begin{pmatrix}
 0&1&\beta&\alpha\\
 1&0&\alpha&\beta\\
 \beta&\alpha&0&1\\
 \alpha&\beta&1&0
 \end{pmatrix}.
\end{equation}
The four mutually orthogonal sign vectors diagonalize this matrix.  More
explicitly,
\begin{equation}\label{eq:sign-diagonalization}
\begin{array}{c|c}
\text{eigenvector}&\text{eigenvalue}\\ \hline
(1,1,1,1)&1+\alpha+\beta\\
(1,1,-1,-1)&1-\alpha-\beta\\
(1,-1,1,-1)&-1-\alpha+\beta\\
(1,-1,-1,1)&-1+\alpha-\beta.
\end{array}
\end{equation}
Using $\alpha+\beta=e^{(\vartheta+1)s}$ and
$\alpha-\beta=e^{(\vartheta-1)s}$ gives, in the same order, the four
numbers in \eqref{eq:intro-spectrum}.  If
$-1\leq\vartheta\leq1$, the first is positive, the second and fourth are
nonpositive, and the third is strictly negative.  Since every coefficient
of \eqref{eq:local-symbol} is nonnegative, the degree-two Hessian criterion
shows that $\Psi_{\vartheta,s}$ is Lorentzian.

It remains to check the symbol on the actual bounded space, rather than only
on the two active variables.  Let $v$ be paired with $y$, and pair $u_r$
with $z_r$.  The identity on polynomials of $y$-degree at most $n$ has
symbol
$\sum_{q=0}^n\binom nq y^qv^{n-q}=(y+v)^n$.  The identity on a spectator
variable $z_r$ of degree at most one has symbol $z_r+u_r$.  Since symbols
factor across disjoint variable sets, the full bounded-degree symbol is
\begin{equation}\label{eq:full-symbol}
 (y+v)^n\prod_{r\notin e}(z_r+u_r)\Psi_{\vartheta,s}.
\end{equation}
Every displayed linear factor is a stable polynomial with nonnegative
coefficients and hence is Lorentzian.  Products preserve Lorentzianity, so
\eqref{eq:full-symbol} is Lorentzian.  Theorem~3.2 of
\cite{branden-huh-2020}, applied with the bound
\eqref{eq:kappa-prelim}, completes the proof.
\end{proof}

\begin{remark}[Endpoints and the sharp local range]\label{rem:local-range}
At $s=0$, the gate is the identity for every $\vartheta$ and the eigenvalues
in \eqref{eq:intro-spectrum} are $2,0,-2,0$.  For $s>0$, the signature test
holds exactly on $[-1,1]$.  At $\vartheta=-1$ the second eigenvalue vanishes;
at $\vartheta=1$ the fourth vanishes.  If $\vartheta<-1$, then
$1-e^{(\vartheta+1)s}>0$, so the first two eigenvalues are positive.  If
$\vartheta>1$, then $e^{(\vartheta-1)s}-1>0$, again producing two positive
eigenvalues.  Thus the local symbol is not Lorentzian outside $[-1,1]$ at
positive time.  This is a limitation of this certificate, not a
counterexample to spectral concavity.
\end{remark}

\section{The global Lorentzian heat flow}

The all-level initial polynomial is
\begin{equation}\label{eq:initial-polynomial}
 \Phi_0(y,z)=\prod_{i=1}^n(y+z_i)
 =\sum_{S\subseteq V}y^{n-|S|}z^S.
\end{equation}
It is homogeneous of degree $n$, lies within the coordinatewise bounds
\eqref{eq:kappa-prelim}, and is stable because it is a product of real
stable linear forms.  Proposition~2.2 of \cite{branden-huh-2020} therefore
shows that it is Lorentzian.  More concretely, its $z$-degree-$k$ part is
$y^{n-k}\sum_{|S|=k}z^S$, which corresponds to the all-ones vector on level
$k$.  Set
\begin{equation}\label{eq:global-generator}
 L^{(\vartheta)}=\sum_{e\in E}w_eL_e^{(\vartheta)},
 \qquad
 \Phi_{\vartheta,t}=e^{tL^{(\vartheta)}}\Phi_0.
\end{equation}

\begin{lemma}[Global heat flow]\label{lem:global}
For every $t\geq0$ and $-1\leq\vartheta\leq1$, the polynomial
$\Phi_{\vartheta,t}$ is Lorentzian.
\end{lemma}

\begin{proof}
If $E$ is empty, then $L^{(\vartheta)}=0$ and the assertion follows directly
from \eqref{eq:initial-polynomial}.  Otherwise list the edges as
$e_1,\ldots,e_m$.  All operators act on the finite-dimensional space of
polynomials satisfying \eqref{eq:kappa-prelim}.  The Lie--Trotter product
formula, in any matrix norm on that space, gives
\begin{equation}\label{eq:trotter}
 e^{tL^{(\vartheta)}}=
 \lim_{r\to\infty}
 \left(
  e^{tw_{e_1}L_{e_1}^{(\vartheta)}/r}\cdots
  e^{tw_{e_m}L_{e_m}^{(\vartheta)}/r}
 \right)^r.
\end{equation}
For a zero-weight edge the corresponding factor is the identity.  Every
other factor has local time $tw_{e_j}/r\geq0$ and preserves Lorentzian
polynomials by Lemma~\ref{lem:local}.
Consequently, their composition and its $r$th power also preserve
Lorentzian polynomials.  Applying this power to $\Phi_0$ therefore produces a Lorentzian
polynomial for every $r$.  Norm
convergence of the operators implies that these polynomials converge
coefficientwise to $\Phi_{\vartheta,t}$.  By
\cite[Theorem~2.25]{branden-huh-2020}, the fixed-degree Lorentzian cone is
closed, so the limit is Lorentzian.  Notice that no commutativity among the
edge generators is assumed.
\end{proof}

The Trotter step is the only passage from local to global.  It is also where
nonnegative weights enter essentially: they ensure that every local time in
\eqref{eq:trotter} is nonnegative, as required by the local preserver lemma.
No connectedness hypothesis appears.

The generator preserves $z$-degree.  For each $k$, the coefficient vector of
$\sum_{|S|=k}z^S$ is $\one_k$, and the restriction identity
\eqref{eq:slice-identification} exponentiates to
\begin{equation}\label{eq:slice-semigroup}
 \left.e^{tL^{(\vartheta)}}\right|_{\deg k}
 =e^{tH_k^{(\vartheta)}}.
\end{equation}
Combining this with \eqref{eq:initial-polynomial} gives the exact all-level
expansion
\begin{equation}\label{eq:Phi-expansion}
 \Phi_{\vartheta,t}(y,z)
 =\sum_{k=0}^n\sum_{|S|=k}
 \left(e^{tH_k^{(\vartheta)}}\one_k\right)_S
 y^{n-k}z^S.
\end{equation}

\section{Finite-time ultra log-concavity and symmetry}

Identify all occupation variables with a single variable $x$.  This is a
nonnegative linear substitution, so Theorem~2.10 of
\cite{branden-huh-2020} and Lemma~\ref{lem:global} show that
\begin{equation}\label{eq:bivariate}
 P_{\vartheta,t}(y,x)
 =\Phi_{\vartheta,t}(y,x,\ldots,x)
 =\sum_{k=0}^n a_k^{(\vartheta)}(t)y^{n-k}x^k
\end{equation}
is Lorentzian.  Indeed, after substitution, the coefficient of
$y^{n-k}x^k$ in \eqref{eq:Phi-expansion} is
\begin{equation}\label{eq:coefficient-sum}
 \sum_{|S|=k}\left(e^{tH_k^{(\vartheta)}}\one_k\right)_S
 =\one_k^{\!*}e^{tH_k^{(\vartheta)}}\one_k
 =a_k^{(\vartheta)}(t).
\end{equation}

We make positivity explicit because it removes the no-internal-zero
qualification in the bivariate Lorentzian criterion.  Each
$H_k^{(\vartheta)}$ is a Metzler matrix: its off-diagonal entries are
nonnegative, although its diagonal entries may be negative when
$\vartheta<0$.  Choose $c$ so that $H_k^{(\vartheta)}+cI$ is entrywise
nonnegative.  The power series for its exponential gives
\begin{equation}\label{eq:metzler-exponential}
 e^{tH_k^{(\vartheta)}}
 =e^{-ct}e^{t(H_k^{(\vartheta)}+cI)}\geq0
 \qquad\text{entrywise}.
\end{equation}
Independently, symmetry of $H_k^{(\vartheta)}$ makes its exponential positive
definite, because all eigenvalues of the exponential are strictly positive.
Thus $a_k^{(\vartheta)}(t)>0$ for every $k$ and $t\geq0$, including the
one-dimensional boundary levels.

Example~2.26 of \cite{branden-huh-2020} applied to
\eqref{eq:bivariate} now says that
$a_k^{(\vartheta)}(t)/\binom nk$ is log-concave.  Written at an interior
index, this is precisely \eqref{eq:ulc}, proving
Theorem~\ref{thm:heat}.  At $t=0$ one has
$a_k^{(\vartheta)}(0)=\binom nk$, so all normalized values equal one and the
inequalities are equalities.

The finite-time sequence also has the same complement symmetry as the token
matrices.  By \eqref{eq:complement-matrices}, functional calculus gives
$e^{tH_{n-k}^{(\vartheta)}}=C_ke^{tH_k^{(\vartheta)}}C_k^{\!*}$, and hence
\begin{equation}\label{eq:heat-complement}
 a_{n-k}^{(\vartheta)}(t)
 =\one_k^{\!*}C_k^{\!*}
   \bigl(C_ke^{tH_k^{(\vartheta)}}C_k^{\!*}\bigr)C_k\one_k
 =a_k^{(\vartheta)}(t).
\end{equation}
Hence the normalized heat contents are symmetric and log-concave, so they
are nondecreasing up to the midpoint: apply the same forward-difference
argument used below to their logarithms.  This extra finite-time monotonicity
is only a reformulation of Theorem~\ref{thm:heat} and complement symmetry;
the spectral result still requires the large-time argument below.

\section{The spectral limit and midpoint monotonicity}

To extract a top eigenvalue from the heat content, we must know that the
all-ones vector sees the top eigenspace.  This is automatic for an irreducible
nonnegative adjacency matrix, but our setting includes negative diagonal
entries, disconnected token graphs, and reducible top eigenspaces.  The
following shifted argument covers all of them.

Fix $k$ and $\vartheta$.  Choose
\begin{equation}\label{eq:perron-shift}
 c\geq \max_{S\in\binom Vk}\{-\vartheta d_k(S),0\},
 \qquad B_k=H_k^{(\vartheta)}+cI.
\end{equation}
Then $B_k$ is symmetric and entrywise nonnegative.  Perron--Frobenius,
without an irreducibility assumption, supplies a nonzero vector $q_k\geq0$
for the spectral radius of $B_k$.  Because a symmetric nonnegative matrix
has its spectral radius as its largest eigenvalue, $q_k$ is a top eigenvector
of $B_k$, and hence of $H_k^{(\vartheta)}$.  Moreover
$\langle\one_k,q_k\rangle>0$.  If several connected components of the token
graph attain the same top eigenvalue, this argument simply chooses a
nonnegative top vector supported on one or more maximizing components.

Let $(\lambda_{k,j},v_{k,j})_j$ be an orthonormal spectral resolution of
$H_k^{(\vartheta)}$.  Then
\begin{equation}\label{eq:spectral-expansion}
 a_k^{(\vartheta)}(t)
 =\sum_j e^{t\lambda_{k,j}}
  |\langle\one_k,v_{k,j}\rangle|^2.
\end{equation}
The shifted Perron vector shows that the orthogonal projection of $\one_k$
onto the top eigenspace is nonzero.  Thus at least one coefficient multiplying
$e^{th_k^{(\vartheta)}}$ in \eqref{eq:spectral-expansion} is positive.  All
summands are nonnegative, so elementary finite-sum asymptotics give
\begin{equation}\label{eq:growth-rate}
 \lim_{t\to\infty}\frac1t\log a_k^{(\vartheta)}(t)
 =h_k^{(\vartheta)}.
\end{equation}
This includes the boundary levels, where $a_0(t)=a_n(t)=1$ and the growth
rate is zero.

Taking logarithms in \eqref{eq:ulc} gives
\begin{align}\label{eq:log-ulc}
 2\log a_k^{(\vartheta)}(t)
 &\geq \log a_{k-1}^{(\vartheta)}(t)
       +\log a_{k+1}^{(\vartheta)}(t) \notag\\
 &\quad +2\log\binom nk-
       \log\binom n{k-1}-\log\binom n{k+1}.
\end{align}
Divide by $t$ and let $t\to\infty$.  The binomial terms are independent of
$t$ and vanish after division, while \eqref{eq:growth-rate} applies to all
three heat contents.  The result is exactly
\eqref{eq:spectral-concavity}.

Equation~\eqref{eq:complement-matrices} shows that $H_k^{(\vartheta)}$ and
$H_{n-k}^{(\vartheta)}$ are permutation-similar.  Hence
$h_k^{(\vartheta)}=h_{n-k}^{(\vartheta)}$.  Put
$d_r=h_r^{(\vartheta)}-h_{r-1}^{(\vartheta)}$.  Discrete concavity makes
$d_r$ nonincreasing in $r$, while symmetry gives
\begin{equation}\label{eq:difference-symmetry}
 d_{n-r+1}=h_{n-r+1}^{(\vartheta)}-h_{n-r}^{(\vartheta)}
 =h_{r-1}^{(\vartheta)}-h_r^{(\vartheta)}=-d_r.
\end{equation}
If $r\leq n-r+1$, then monotonicity of the differences yields
$d_r\geq d_{n-r+1}=-d_r$, and hence $d_r\geq0$.  Taking $r=k+1$ proves
\eqref{eq:monotonicity} whenever $k<n/2$ and completes the proof of
Theorem~\ref{thm:spectral}.  For odd $n$, the two middle levels are
complements and have equal top eigenvalue.

\begin{corollary}[APS Conjectures 5 and 6]\label{cor:aps}
For every nonnegatively weighted graph and every $1\leq k<n/2$,
\begin{align*}
 \lambda_{\max}\!\left(Q(F_k(G))\right)
 &\leq\lambda_{\max}\!\left(Q(F_{k+1}(G))\right),\\
 \lambda_{\max}\!\left(A(F_k(G))\right)
 &\leq\lambda_{\max}\!\left(A(F_{k+1}(G))\right).
\end{align*}
In particular, these are the inequalities in Conjectures~5 and~6 of
\cite{apte-parekh-sud-2026}.
\end{corollary}

\begin{proof}
Apply \eqref{eq:monotonicity} with $\vartheta=1$ and $\vartheta=0$,
respectively, and use $H_k^{(1)}=Q(F_k(G))$ and
$H_k^{(0)}=A(F_k(G))$.
\end{proof}

\section{Top spectral-projection polynomials}

The multivariate heat flow also controls the top eigenspaces within each
token level, not merely their eigenvalues.  Fix $k$ and define
\begin{equation}\label{eq:fixed-level}
 f_{k,t}(z)=\sum_{|S|=k}
 \left(e^{tH_k^{(\vartheta)}}\one_k\right)_S z^S.
\end{equation}
Equation~\eqref{eq:Phi-expansion} gives the exact extraction formula
\begin{equation}\label{eq:extraction}
 f_{k,t}(z)=\frac1{(n-k)!}
 \left.\partial_y^{n-k}\Phi_{\vartheta,t}(y,z)\right|_{y=0}.
\end{equation}
Indeed, differentiating $n-k$ times kills terms with smaller $y$-degree,
and setting $y=0$ kills those with larger $y$-degree; the surviving term
acquires the factor $(n-k)!$.  Repeated differentiation preserves
Lorentzianity by \cite[Corollary~2.11]{branden-huh-2020}, and the
specialization $y=0$ is a nonnegative linear substitution covered by
\cite[Theorem~2.10]{branden-huh-2020}.  Therefore $f_{k,t}$ is Lorentzian.
Its coefficients are nonnegative by \eqref{eq:metzler-exponential}, and the
polynomial is nonzero.

Let $\Pi_k$ be the orthogonal projection onto the top eigenspace of
$H_k^{(\vartheta)}$.  The spectral theorem yields
\begin{equation}\label{eq:projection-limit}
 e^{-th_k^{(\vartheta)}}e^{tH_k^{(\vartheta)}}\one_k
 \longrightarrow \Pi_k\one_k.
\end{equation}
The right side is nonzero by the shifted Perron argument in
\eqref{eq:perron-shift}.  Applying the coefficient-vector identification to
\eqref{eq:projection-limit} shows that
\begin{equation}\label{eq:projection-polynomial-limit}
 e^{-th_k^{(\vartheta)}}f_{k,t}(z)
 \longrightarrow
 \sum_{|S|=k}(\Pi_k\one_k)_S z^S
\end{equation}
coefficientwise.  Every polynomial on the left is a positive multiple of a
Lorentzian polynomial.  Theorem~2.25 of \cite{branden-huh-2020} permits the
limit and gives the following consequence.

\begin{corollary}\label{cor:perron-polynomial}
For every $k$ and $-1\leq\vartheta\leq1$, the nonzero polynomial
\begin{equation}\label{eq:perron-polynomial}
 \sum_{|S|=k}(\Pi_k\one_k)_S z^S
\end{equation}
is Lorentzian.  If the top eigenspace is one-dimensional, the polynomial of
a nonnegative Perron eigenvector is Lorentzian, up to positive scaling.
\end{corollary}

The projection formulation is essential when the top eigenvalue is
multiple.  It selects the canonical nonnegative vector $\Pi_k\one_k$ rather
than asserting that every vector in a reducible top eigenspace has
Lorentzian coefficients.  At $k=0$ and $k=n$ the conclusion reduces to a
positive constant or a positive multiple of $z_1\cdots z_n$, respectively.
At $\vartheta=-1$, the all-ones vector already lies in the zero eigenspace of
the Laplacian on every connected component, so $\Pi_k\one_k=\one_k$ even
when $F_k(G)$ is disconnected.

Theorem~\ref{thm:spectral} concerns the largest eigenvalues of undirected
token graphs with nonnegative edge weights.  Signed weights, directed token
graphs, other exclusion rules, and the remaining eigenvalues are outside its
scope.

In the Hamiltonian interpretation of APS \cite{apte-parekh-sud-2026}, the
matrices $H_k^{(\vartheta)}$ occur as fixed-particle-number blocks of
number-conserving $2$-local models.  Theorem~\ref{thm:spectral} says that the
top block energy is concave in particle number and, by complement symmetry,
is maximized at half filling: at $k=n/2$ for even $n$, and at the two equal
middle levels for odd $n$.

\bibliographystyle{plain}
\bibliography{references}

\end{document}